\documentclass[12pt,a4paper]{article}
\usepackage[english]{babel}
\usepackage{amsmath,amssymb,graphicx}
\usepackage{fullpage} 
\usepackage{hyperref}
\usepackage{mathtools}
\numberwithin{equation}{section}
\usepackage[utf8]{inputenc}
\usepackage{amsfonts}
\usepackage{amsthm}
\usepackage{xcolor}
\usepackage{enumerate}
\usepackage{cleveref}

\def\cH{\mathcal{H}}
\def\cK{\mathcal{K}}\def\cC{\mathcal{C}}
\def\cN{\mathcal{N}}

\def\cS{\mathcal{S}}\def\cD{\mathcal{D}}

\def\bR{\mathbb{R}}

\def\bE{\mathbb{E}}
\def\bP{\mathbb{P}}

\def\fX{\mathbf{X}}
\def\fB{\mathbf{B}}
\def\fY{\mathbf{Y}}

\theoremstyle{plain}
\newtheorem{theorem}{Theorem}[section]
\newtheorem{proposition}[theorem]{Proposition}
\newtheorem{lemma}[theorem]{Lemma}
\newtheorem{corollary}[theorem]{Corollary}

\theoremstyle{definition}
\newtheorem{definition}[theorem]{Definition}

\theoremstyle{remark}
\newtheorem{remark}[theorem]{Remark}

\title{Approximate Transitivity of Young Translation on Rough Paths}

\date{}

\author{Carlo Bellingeri  \thanks{ \texttt{carlo.bellingeri@uha.fr}  IRIMAS, Université de Haute-Alsace, Mulhouse, France}\and Paul Gassiat\thanks{\texttt{paul.gassiat@univ-eiffel.fr} LAMA, Université Gustave Eiffel, Marne-la-Vallée, France \& IUF}
\and Hugo Nouaille \thanks{\texttt{hugo.nouaille@univ-lorraine.fr}  IECL, Université de Lorraine, Vandœuvre-lès-Nancy, France}
}

\begin{document}
 \medskip
 
 \maketitle
 
 \begin{abstract}
\noindent We show that Young translation has dense orbits in the space of rough paths: for any two geometric rough paths, one can translate the first by a sequence of smooth paths so that it converges to the second in rough path topology. As applications, we obtain full-support criteria for rough paths arising from random series and Gaussian processes, including non-centered fractional Brownian rough paths.
\end{abstract}


\medskip



\section{Introduction}

One of the most elementary approaches to proving support theorems for stochastic
processes relies on two ingredients: quasi-invariance under deterministic shifts and
density of admissible shifts. For example, let $S = \operatorname{supp} \mu$ be the
support of the law $\mu$ of a random variable $X$ taking values in a separable Banach
space $E$. Assume that there exists a dense subset $H \subset E$ such that for every
$h \in H$ the law of $X + h$ is equivalent to $\mu$. Then necessarily $S = E$.
Indeed, quasi-invariance implies that for every $h \in H$, $S + h = S$. Fix $x \in S$
and $y \in E$. Since $H$ is dense in $E$, there exists a sequence $(h^n) \subset H$
such that $x + h^n \to y$. For each $n$, the identity $S + h^n = S$ implies
$x + h^n \in S$. Since $S$ is closed, passing to the limit yields $y \in S$, and
hence $E = S$.

When studying support theorems for differential equations of the form
\[
    \mathrm{d}Y_t = \sum_{i=1}^d V_i(Y_t)\,\mathrm{d}X^i_t,
    \qquad Y_0 = y \in \mathbb{R}^e,
\]
with bounded smooth vector fields $V_i : \mathbb{R}^e \to \mathbb{R}^e$ and a driving
path $X : [0,1] \to \mathbb{R}^d$, rough path theory (see
\cite{lyons1998, frizbook, FH20}) shows that the natural object governing the support
of $Y$ is a rough path lift $\mathbf{X}$ of the process $X$. For a smooth path $X$,
the canonical lift is given by the collection of iterated integrals
\begin{equation}\label{eq:p_lift}
    S_{s,t}^{N}(X) = \left(
        X_t - X_s,\;
        \ldots,\;
        \int_{s \leq s_1 \leq \cdots \leq s_{N} \leq t}
        \mathrm{d}X_{s_1} \otimes \cdots \otimes \mathrm{d}X_{s_{N}}
    \right),
\end{equation}
for any integer $N\geq 1$. The rough path lift $\mathbf{X}$ of a stochastic process $X$ is then constructed as a suitable limit of such lifts along smooth approximations at some fixed order $N$, with
the iterated stochastic integrals defined in an appropriate sense. This enhancement restores continuity of the solution
map for differential equations driven by irregular signals. Consequently, support
theorems in rough path topologies provide a natural framework for extending the
classical Stroock--Varadhan support theorem \cite{SV72}. This approach has proved
especially effective for Gaussian rough paths, where the structure of the associated
Cameron--Martin space can be exploited to characterize the support explicitly; see
for instance \cite{LQZ02, FV10gaussian}.

In rough path topology, however, the preceding linear support argument cannot be
transplanted directly. A first issue is that rough path space is not a linear space,
so addition is not available. The natural replacement in this context is given by
Young translation \cite{frizbook}: given a rough path $\mathbf{X}$ and a sufficiently
regular path $h$, one can define the Young translation $\mathbf{X} \boxplus h$, a new rough path obtained by summing the paths at the first level and then propagating appropriately at the
higher levels. When $\mathbf{X}=S^N(X)$ for some smooth path $X$, this reduces to $S^N(X + h)$, see section \ref{sec:2}. The notation varies in the literature; we follow \cite{Bel22}, where
generalisations of sums of geometric rough paths are studied.

In order to implement the simple support argument above, the missing ingredient is a density result on orbits of Young translation. This is precisely the main result of this article in Theorem \ref{thm_main_result}, which can be stated in simple terms as follows:
\begin{theorem}\label{thm_main}
	Let $\mathbf X$, $\mathbf Y$ be two arbitrary $\alpha$-H\"older geometric rough paths over $\mathbb R^d$. Then, there exists a sequence of paths $(h^n)\subset C^{\infty}([0,1],\mathbb R^d)$ such that 
	\begin{equation}\label{eq:main_conv}
(\mathbf X\boxplus h^n)\rightarrow \mathbf Y
	\end{equation}
	 in  $\alpha$-H\"older rough path distance.
	 \end{theorem}

Let us emphasize that the nonlinearity of rough path space creates a genuine 
difficulty. Consider the case where one wants to translate $\mathbf{X}$ to 
$\mathbf{Y} = \mathbf{1}$, the constant rough path. Recall that any 
$\alpha$-H\"older geometric rough path over $\mathbb{R}^d$ can be obtained as 
the limit of smooth lifts $S^{\lfloor 1/\alpha \rfloor}(x^n) = \mathbf{1} 
\boxplus x^n$ for some sequence $(x^n) \subset C^{\infty}([0,1], \mathbb{R}^d)$. 
In analogy with the linear case, a natural candidate would be to translate 
$\mathbf{X}$ along $(-x^n)$. However, in contrast with the linear case, 
\[
S^{\lfloor 1/\alpha \rfloor}(x^n) = \mathbf{1} \boxplus x^n \to \mathbf{X}
\;\;\not\Rightarrow\;\;
\mathbf{X} \boxplus (-x^n) \to \mathbf{1}\,.
\]

This is illustrated by the following example (taken from \cite{Coutin07}). 
Consider the pure area rough path $\mathbf{X}$ corresponding to the 
approximating sequence
\[
x^n_t = n^{-1/2} \begin{pmatrix} \cos(nt) \\ \sin(nt) \end{pmatrix}.
\]
Then (see e.g.\ \cite{frizbook}, Exercise~2.10),
\[
S^2(x^n) \to \mathbf{X}, \qquad \mathbf{X}_{s,t} = \left(0,\,\tfrac{1}{2}(t-s)A\right),
\]
where
\[
A = \begin{pmatrix} 0 & 1 \\ -1 & 0 \end{pmatrix}.
\]
Since $\mathbf{X}$ has vanishing first-level component, by the explicit identity \eqref{eq:coordinates_Young} one has the uniform convergence
\begin{align*}
(\mathbf{X} \boxplus (-x^n))_{s,t}
&= \left(x^n_{s}- x^n_{t},\; \tfrac{1}{2}(t-s)A + \int_s^t (x^n_{u}- x^n_{s}) \otimes \mathrm{d}x^n_u \right) \to  \left(0,\; (t-s)A \right),
\end{align*}
where we used that $x^n_{s,t} \to 0$  and $\int_s^t x^n_{s,u} \otimes \mathrm{d}x^n_u 
\to \tfrac{1}{2}(t-s)A$ uniformly. In other words, rather than  converging to the constant path $\mathbf{1}$, the translated rough path  converges to another pure area rough path with area accumulating at double 
the rate.

However, if one considers the paths approximating the opposite pure area
\[
{h}^n_t=n^{-1/2}\begin{pmatrix}
	\cos( - n t)\\
	\sin(- n t)
\end{pmatrix},
\]
then the same computation as above yields the uniform convergence
\begin{align*}
(\mathbf{X} \boxplus (-h^n))_{s,t}
&= \left(h^n_{s}- h^n_{t},\; \tfrac{1}{2}(t-s)A + \int_s^t (h^n_{u}- h^n_{s}) \otimes \mathrm{d}h^n_u \right)  \to  \mathbf{1}\,,
\end{align*}
and it is also straightforward to check that the convergence holds in rough path topology. Hence in this case, we are still able to find an explicit sequence of smooth paths $(k^n)$, such that $\mathbf{X} \boxplus k^n \to \mathbf 1$.

In the case of a general rough path, it is not clear how to define explicitly such a suitable sequence. Instead, we will rely on an abstract controllability result from \cite{Gassiat24}, which implies that we can find smooth paths such that, on arbitrarily fine intervals, the rough path translation has trivial increments. By a careful control of the $1$-variation norm of these paths, we are able to push this to convergence in rough path metric.

As an application of our main theorem, we then obtain support criteria for rough-path-valued random variables. If \(F\) is a nonempty closed subset of rough path space which is stable under Young translations by a dense class of controls, then the density of Young-translation orbits forces \(F\) to be the whole rough path space. Applied to the support of a random rough path, this gives full support once one has verified stability under a sufficiently rich family of deterministic translations. We use this principle for random series and Gaussian rough paths, including non-centered fractional Brownian rough paths.

\begin{remark}
It is natural to ask whether Theorem~\ref{thm_main} extends to non-geometric 
rough paths, such as $\mathcal{H}$-rough paths over a Hopf algebra $\mathcal{H}$ 
in the sense of \cite{Tapia2018}. Our approach does not directly carry over to 
this setting: the key obstruction is that Chow's theorem 
\cite[Theorem~7.28]{frizbook} (see also Theorem~\ref{thm_chow_transl}), which 
is central to our argument, fails to hold in general for non-geometric rough 
paths. More precisely, while one can define lift maps analogous to 
\eqref{eq:p_lift} in this setting, surjectivity onto the full space of rough 
paths is lost, because the iterated integrals of a path must satisfy additional 
algebraic relations imposed by the Hopf algebra structure.

A partial remedy is available when the governing Hopf algebra $\mathcal{H}$ is 
\emph{cofree}, i.e.\ isomorphic as a coalgebra to the cofree cocommutative 
coalgebra on its primitive elements. As shown in \cite{Bel26} in the case of 
branched rough paths, any $\mathcal{H}$-rough path $\mathbf{X}$ over a 
finite-dimensional vector space $V$ can be canonically identified with an 
anisotropic geometric rough path $\widehat{\mathbf{X}}$ over an enlarged space 
$\widehat{V} \supset V$ still finite-dimensional. An analogue of Theorem~\ref{thm_main} then holds 
naturally for $\widehat{\mathbf{X}}$ in the anisotropic geometric setting 
without essential difficulty, by taking approximating paths 
$(h^n) \subset C^{\infty}([0,1], \widehat{V})$.

The remaining challenge is to express the translation operator 
$\widehat{\mathbf{X}} \boxplus h^n$ back in the original coordinates. This 
requires starting from paths $(h^n) \subset C^{\infty}([0,1], \widehat{V})$ 
living in the full enlarged space and defining a suitable notion of rough 
differential equation in the extra coordinates, whose coefficients encode 
the Hopf algebra structure of $\mathcal{H}$, in order to give meaning to 
the translation in the original coordinates. We expect this to be essentially straightforward, amounting to careful but routine bookkeeping of the Hopf 
algebra structure, and it should cover both quasi-geometric rough paths \cite{Bel23} and branched rough paths as special cases.
\end{remark}

\begin{remark}
While it is not always true that a given geometric rough path $\mathbf X$ admits a sequence $(x^n)$ such that both $S^{\lfloor 1/\alpha \rfloor}(x^n) \to \mathbf X$ and $\mathbf{X} \boxplus (-x_n) \to \mathbf 1$, it is known to hold in many cases of interest. This is the case if $(x^n)$ is a good rough path sequence in the sense of \cite{Coutin07}, namely if $(x^n, \mathbf X) \to_n (\mathbf X, \mathbf X)$ in rough path metric. In particular, this is true almost surely if $\mathbf X$ is the lift of a (fractional) Brownian motion, or more generally of Gaussian processes satisfying the assumptions of \cite{FV10gaussian}. Proving such a convergence is the crucial step in the proofs of existing rough path support theorems, cf e.g. \cite[Section 15.8]{frizbook}. The main interest of our result is that it allows to bypass this step, at least in the case where the Cameron-Martin space of the Gaussian process is known to be dense.
\end{remark}

The paper is organised as follows. In Section~\ref{sec:2}, we introduce 
the necessary preliminaries and notation, and we prove a simple local surjectivity result in the context of maps with values in the step-\(N\) free nilpotent Lie group. Section~\ref{sec:3} is devoted to a precise statement and proof of Theorem~\ref{thm_main}. In Section~\ref{sec:4}, we apply Theorem~\ref{thm_main} to derive a new criterion for support  theorems in rough path theory.


\subsection*{Acknowledgements}
C.B. gratefully acknowledges funding from the Université de Haute-Alsace through the research program “Contrat de Bienvenue”.  H.N. gratefully acknowledges support from the ERC Starting Grant Low Regularity Dynamics via Decorated Trees (LoRDeT). The views and opinions expressed are, however, those of the authors only and do not necessarily reflect those of the European Union or the European Research Council.

\section{Preliminaries}\label{sec:2}
\subsection{Paths, Rough paths and Young translations}
We briefly recall here the main properties of \emph{paths}, \emph{geometric rough paths} and their translation that will be used throughout the paper. More details can be found in \cite{frizbook}.

In the whole paper, all paths will be functions $h \colon [0,1] \to \bR^d$ starting from $0$, i.e. with $h(0)=0$. For any such path $h$, we denote its \emph{increments} by $h_{s,t} := h_t - h_s$ for $s,t \in [0,1]$. For $p \geq 1$, we denote by $C^{p\text{-var}}_0$ the set of continuous paths starting from $0$ such that
\[
\| h \|_{p\text{-var}}^p := \sup_{0 = t_0 \leq \ldots \leq t_m = 1}
    \sum_{i=0}^{m-1} |h_{t_i,t_{i+1}}|^p < \infty\,,
\]
where $|\cdot|$ is the usual euclidean norm on $\mathbb{R}^d$. The quantity $\| h \|_{p\text{-var}}$ is called the \emph{$p$-variation norm} of $h$. This definition can be restricted to any subinterval $[s,t] \subset [0,1]$, in which case we write $\| h \|_{p\text{-var};[s,t]}$. For $0 < \alpha \leq 1$, we define $C^{\alpha}_0$ as the space of \emph{$\alpha$-H\"older continuous} paths starting from $0$ and such that
\[
\| h \|_{\alpha} := \sup_{s \neq t} \frac{|h_{s,t}|}{|t-s|^{\alpha}} < \infty.
\]
Clearly, $C^{\alpha}_0 \subset C^{(1/\alpha)\text{-var}}_0$ with continuous inclusion. Both structures can be combined: for $0 < \alpha \leq 1$ and $q \geq 1$ satisfying $\alpha + \frac{1}{q} > 1$, we define $C^{\alpha, q\text{-var}}_0$ as the set of continuous paths starting from $0$ such that
\[
\| h \|_{\alpha,q\text{-var}} :=
    \sup_{s \neq t} \frac{\| h \|_{q\text{-var};[s,t]}}{|t-s|^{\alpha}} < \infty\,.
\]
We also let $C^{0;\alpha,q\text{-var}}_0$ be the closure of $C^{\infty}_0$ functions in $C^{\alpha, q\text{-var}}_0$ (i.e. for the above norm). We note that it contains $C^{\alpha', q'\text{-var}}_0$ for any $\alpha'>\alpha, q'<q$. 

All these spaces can be extended to the context of geometric rough paths by considering paths with values in the step-\(N\) free nilpotent Lie group \((G^N(\mathbb{R}^d), \otimes)\), for some integer \(N \geq 1\), see \cite[Theorem 7.30]{frizbook}. We will denote by \(\mathbf{1}\) its unit element; for \(\lambda > 0\), we will also use \(\delta_\lambda \colon G^N(\mathbb{R}^d) \to G^N(\mathbb{R}^d)\) its intrinsic dilation and by \(|\cdot|\) with a slight abuse of notation  the Carnot--Carath\'eodory norm on \(G^N(\mathbb{R}^d)\) (see \cite[Theorem 7.32]{frizbook}), which satisfies $|\delta_\lambda(g)| = \lambda |g| $ for all $g \in G^N(\mathbb{R}^d)$, $\lambda > 0$.

In this context for any path $\fX \colon [0,1] \to G^N(\bR^d)$, we define its increments as
\[
\fX_{s,t} := \fX_s^{-1} \otimes \fX_t, \qquad s,t \in [0,1].
\]
 and we define for any $p\geq1$ the space of \emph{weakly geometric $p$-variation rough paths} $\cC^{p\text{-var}}_{\mathbf{1}}$, as the space  of paths $\fX \colon [0,1] \to G^{\lfloor p \rfloor}(\bR^d)$ starting at ${\mathbf{1}}$ and such that
\[
\| \fX \|_{p\text{-var}}^p :=
    \sup_{0 = t_0 \leq \ldots \leq t_m = 1}
    \sum_{i=0}^{m-1} |\fX_{t_i,t_{i+1}}|^p < \infty.
\]
Similarly, for $0<\alpha \leq1$ we define $\cC^{\alpha}_{\mathbf{1}} \subset \cC^{1/\alpha\text{-var}}_{\mathbf{1}}$ as the space of \emph{weakly geometric $\alpha$-H\"older rough paths}, i.e. $G^{\lfloor 1/ \alpha\rfloor}(\bR^d)$-valued  paths $\fX$ such that
\[
\sup_{s \neq t} \frac{|\fX_{s,t}|}{|t-s|^{\alpha}} < \infty.
\]
Two geometric rough paths $\fX, \fY$ can be compared using the inhomogeneous $p$-variation rough path distance
\[
\rho_{p\text{-var}}(\fX, \fY) :=
\sum_{k=1}^{\lfloor p \rfloor}
    \sup_{0 = t_0 \leq \ldots \leq t_m = 1}
    \left(
        \sum_{i=0}^{m-1}
        \|
            \pi_k\big( \fX_{t_i,t_{i+1}} - \fY_{t_i,t_{i+1}} \big)
        \|^{\frac{p}{k}}
    \right)^{\frac{k}{p}},
\]
where $\pi_k \colon G^N(\bR^d) \to (\bR^d)^{\otimes k}$ denotes the canonical projection onto the $k$-th tensor level and $\|\cdot\|$ is the usual tensorisation of the euclidean norm on $(\bR^d)^{\otimes k}$. The analogous distance for weakly geometric $\alpha$-H\"older rough paths is
\[
\rho_{\alpha}(\fX, \fY) :=
\sum_{k=1}^{\lfloor 1/\alpha \rfloor}
 \sup_{s \neq t} \frac{\|\pi_k(\fX_{s,t} - \fY_{s,t})\|}{|t-s|^{k\alpha}}\,.
\]

In contrast to much of the recent literature, where the term \emph{rough path} typically refers to the weakly geometric variety defined above, we maintain a distinct notion of geometric rough paths. A path $\fX \in \cC^{p\text{-var}}$ is called a \emph{$p$-variation (strictly) geometric rough path} if there exists  a sequence of paths $(x^n)\subset  C^{1\text{-var}}$ such that  
\begin{equation}\label{eq_geom_approx}
 \rho_{p\text{-var}}(S^{\lfloor p \rfloor}(x^n), \fX)\to 0\,,
 \end{equation}
 where $S^{\lfloor p \rfloor}(x)$ denotes the canonical lift of a bounded variation path $x$ to $G^{\lfloor p \rfloor}(\bR^d)$, see \eqref{eq:p_lift}. The space of such paths is denoted $\cC^{0;p\text{-var}}$.  The analogous definition holds in the H\"older setting, yielding the space $\alpha$-H\"older (strictly) geometric rough paths $\cC^{0;\alpha}_g$. That being said, by \cite[Theorem 8.12]{frizbook} any weakly geometric $p$-variation rough path $\mathbf{X}$ can be approximated by a sequence $(x^n)\subset C^{1\text{-var}}_0$  such that
\begin{equation}\label{eq:approx}
\lim_{n\to \infty} S^{\lfloor p \rfloor}(x^n) = \mathbf{X}\,, \text{ uniformly over $[0,1]$}\,,
    \qquad \sup_{n\geq 0} \|S^{\lfloor p \rfloor}(x^n)\|_{p\text{-var}} < +\infty.
\end{equation}

Paths and rough paths can be combined in a standard way; see \cite[Theorem 9.33]{frizbook}. Let \(1 \le q < 2\) and \(p\) be such that \(1/p + 1/q > 1\). There exists a canonical way to add a path \(h \in C^{q\text{-var}}_0\) to a rough path \(\mathbf{X} \in \mathcal{C}^{p\text{-var}}_{\mathbf{1}}\), yielding a new rough path \(\mathbf{X} \boxplus h \in \mathcal{C}^{p\text{-var}}_{\mathbf{1}}\) such that
\[
\pi_1(\mathbf{X} \boxplus h) = \pi_1(\mathbf{X}) + h.
\]
The rough path \(\mathbf{X} \boxplus h\) is called the \emph{Young translation} of \(\mathbf{X}\) by \(h\). Viewing \(\mathbf{X}\) as encoding the iterated integrals of the path \(X := \pi_1(\mathbf{X})\), one obtains the following explicit expression: for each level \(1 \le k \le N\),
\begin{equation}\label{eq:coordinates_Young}
\pi_k(\mathbf{X} \boxplus h)_{s,t}
= \sum_{\varepsilon \in \{0,1\}^k}
\int_{s \le t_1 \le \cdots \le t_k \le t}
\mathrm{d}(X,h)^{\varepsilon_1}_{t_1} \otimes \cdots \otimes \mathrm{d}(X,h)^{\varepsilon_k}_{t_k},
\end{equation}
where, for \(\varepsilon = (\varepsilon_1,\dots,\varepsilon_k) \in \{0,1\}^k\), we use the convention
\[
\mathrm{d}(X,h)^{\varepsilon_i}_t =
\begin{cases}
\mathrm{d}X_t, & \text{if } \varepsilon_i = 0,\\
\mathrm{d}h_t, & \text{if } \varepsilon_i = 1.
\end{cases}
\]
Expression \eqref{eq:coordinates_Young} is only formal and one need to interpret the integrations in a rigorous way. An exact consecutive block of zeros of length \(j\) (say at positions \(i,\dots,i+j-1\)) contributes  the iterated integral \(\mathrm{d}\pi_j(\mathbf{X}_{t_i,t_{i+j-1}})\), taken from the rough path \(\mathbf{X}\) at level \(j\). Likewise, an exact consecutive block of ones of length \(\ell\) at position $k,\dots ,k+l-1$ contributes the iterated integral \(\mathrm{d}\pi_\ell(S^{\lfloor p \rfloor}(h)_{t_k,t_{k+l-1}})\). At the interface between two consecutive blocks the corresponding integrals are understood as Young integrals. These involve an integrand of finite \(p/j\)-variation and an integrator of finite \(q/\ell\)-variation, and are well-defined under the condition \(1/p + 1/q > 1\) and the integration domain which explain the order of integration. An alternative definition of $\mathbf{X} \boxplus h$ can be given by taking the pointwise limit
\begin{equation}\label{eq:pointwise}
(\mathbf{X} \boxplus h)_t=\lim_{n \to \infty} (S^{\lfloor p \rfloor}(x^n)\boxplus h)_t= \lim_{n \to \infty} S_t^{\lfloor p \rfloor}(x^n + h)\, \text{    for all $t \in [0,1]$}
\end{equation}
for any sequence $(x^n)$ of $C^{1\text{-var}}$ elements satisfying \eqref{eq:approx}.

Following again \cite[Theorem 9.33]{frizbook}, the translation operation 
\[(\mathbf{X}, h)\in  \mathcal{C}^{p\text{-var}}_{\mathbf{1}} \times C^{q\text{-var}}_0 \to \mathbf{X} \boxplus h \in \mathcal{C}^{p\text{-var}}_{\mathbf{1}}\]
 is continuous, and in fact locally Lipschitz continuous. That is for any $M > 0$ there exists a constant $C(M) > 0$ such that if 
$\|\mathbf{X}\|_{p\text{-var}}, \|\mathbf{X}'\|_{p\text{-var}}, \|h\|_{q\text{-var}}, \|h'\|_{q\text{-var}} \le M$, then
\begin{equation}\label{eq_loc_Lip}
\rho_{p\text{-var}}(\mathbf{X}\boxplus h, \mathbf{X}'\boxplus h') \le C(M) \bigl( \rho_{p\text{-var}}(\mathbf{X},\mathbf{X}') + \|h-h'\|_{q\text{-var}} \bigr).
\end{equation}
Moreover, this map remains locally Lipschitz continuous when restricted to H\"older rough paths and paths in $C^{\alpha; p\text{-var}}$
\[(\mathbf{X}, h) \in \mathcal{C}^{\alpha}_{\mathbf{1}} \times C^{\alpha,q\text{-var}}_0 \to \mathbf{X} \boxplus h \in \mathcal{C}^{\alpha}_{\mathbf{1}}\,.\] 
In addition to this properties, we can also look at Young translations as solutions of a specific RDE.
\begin{proposition}
Let $1\le q < 2$ and $p \ge 1$ with $1/p + 1/q > 1$. For any path \(h \in C^{q\text{-var}}_0\) and any  rough path \(\mathbf{X} \in \mathcal{C}^{p\text{-var}}_{\mathbf{1}}\), the  Young translation $\mathbf{Y}=\mathbf{X} \boxplus h$ is the unique solution to the tensor valued RDE 
\begin{equation}\label{eq:RDE_translation}
\mathrm{d}\mathbf{Y}_t=  \sum_{i=1}^d  V_i(\mathbf{Y}_t)\mathrm{d}\mathbf{X}^i_t+ \sum_{i=1}^d V_i(\mathbf{Y}_t)\mathrm{d}h^i_t\,,\quad  \mathbf{Y}_0=\mathbf{1}
\end{equation}
where  $V_i(y):=y\otimes e_i$ are the right-invariant vector fields on the free step-$\lfloor p \rfloor$
nilpotent Lie group $G^{\lfloor p \rfloor}(\bR^d)$ and $\{e_i\colon i\in \{1,\ldots d\}\}$ is the canonical basis of $\mathbb{R}^d$. In case when $ \mathbf{X}\in \mathcal{C}^{\alpha}_{\mathbf{1}} $  the result holds by taking $h\in C^{\alpha,q\text{-var}}_0$ with $\alpha+ 1/q>1$.
\end{proposition}
\begin{proof}
Following \cite[Theorem 9.33]{frizbook} and equation \eqref{eq:coordinates_Young}, the rough path $\mathbf{X} \boxplus h$ arises as $F(\mathbf{X} \oplus h)$, where $F\colon G^{\lfloor p \rfloor}(\mathbb{R}^{2d})\to G^{\lfloor p \rfloor}(\bR^d)$ is the unique group homomorphism such that for all $\xi\in \mathbb{R}^{2d}$, $\xi=(\xi_1,\xi_2)$, one has $F(\exp_{2d}((\xi_1,\xi_2)))=\exp_d(\xi_1+\xi_2)$. Here for every $n\geq 1$,  $ \exp_n\colon  \mathbb{R}^n\to G^{\lfloor p \rfloor}(\mathbb{R}^{n})$ is the restriction of the canonical Lie group exponential map and $\mathbf{X} \oplus h\in \mathcal{C}^{p\text{-var}}_{\mathbf{1}}$ is the Young pairing of $\mathbf{X}$ and $h$, a $\mathbb{R}^{2d}$ rough path whose projection on the first $d$ coordinates coincides with $\mathbf{X}$ and the projection on the other $d$ coordinates coincides with $h$, see \cite[Theorem 9.26]{frizbook}. Since $\mathbf{Z}:=\mathbf{X} \oplus h$ trivially satisfies the signature equation in the sense of RDE,
\begin{equation}\label{eq:rde_young_pairing}
\mathrm{d}\mathbf{Z}_t = \sum_{i=1}^{2d} \widetilde{V}_i(\mathbf{Z}_t)\, \mathrm{d}\mathbf{Z}_t^i,
\end{equation}
where $\widetilde{V}_i$ are the right-invariant vector fields on $G^N(\mathbb{R}^{2d})$ corresponding to the canonical basis $\{e_i\}$ of $\mathbb{R}^{2d}$, pushing forward \eqref{eq:rde_young_pairing} via $F$ yields exactly \eqref{eq:RDE_translation}. Uniqueness of the solution to \eqref{eq:RDE_translation} follows from the standard well‑posedness theory for rough differential equations, completing the proof. In the case of a H\"older rough path, the reasoning follows in the same way by starting from  the existence of a Young pairing in this setting \cite[Lemma 9.24,Theorem 9.26]{frizbook}.
\end{proof}

Thanks to the identification provided by the RDE \eqref{eq:RDE_translation}, we can interpret the path \( h \) in \( (\mathbf{X} \boxplus h) \) as a control. This allows us to apply the Chow--Rashevskii theorem for RDEs, which guarantees controllability of the map
\[
(\mathbf{X}, h) \in \mathcal{C}^{p\text{-var}}_{\mathbf{1}} \times C^{q\text{-var}}_0 \;\longmapsto\; (\mathbf{X} \boxplus h)_1 \in G^{\lfloor p \rfloor}(\mathbb{R}^d)\,.
\]
with surjective differential. The precise statement is as follows.

\begin{theorem}[Chow--Rashevskii theorem for translations]\label{thm_chow_transl}
Let \( \mathcal{H} \) be a Hilbert space of \( \mathbb{R}^{d} \)-valued paths such that $
\mathcal{H} \subset C^{q\text{-var}}_0$ with  $1/p + 1/q > 1$ and $C^\infty_0([0,1]; \mathbb{R}^{d}) \subset \mathcal{H}$.
For any  \( \mathbf{X} \in \mathcal{C}^{p\text{-var}}_{\mathbf{1}} \) and any \( g \in G^{\lfloor p \rfloor}(\mathbb{R}^d) \), there exists \( h \in \mathcal{H} \) such that $
(\mathbf{X} \boxplus h)_1 = g.$
Moreover, the map $h \mapsto (\mathbf{X} \boxplus h)_1$ is Fr\'echet differentiable, with gradient \( \nabla_{\mathcal{H}} (\mathbf{X} \boxplus h)_1 \), and \( h \) can be chosen to be non-singular, that is
\begin{equation*}
\langle \nabla_{\mathcal{H}} (\mathbf{X} \boxplus h)_1, \cdot \rangle \colon \mathcal{H} \to T_g G^{\lfloor p \rfloor}(\mathbb{R}^d) \quad \text{is surjective},
\end{equation*}
with  $T_g G^{\lfloor p \rfloor}(\mathbb{R}^d)$ the tangent space of $G^{\lfloor p \rfloor}(\mathbb{R}^d)$ at point $g$.  In case when $ \mathbf{X}\in \mathcal{C}^{\alpha}_{\mathbf{1}} $  and the Hilbert space  satisfies $\mathcal{H}\subset C^{\alpha,q\text{-var}}_0$ with $\alpha+ 1/q>1$ the same result holds.
\end{theorem}
\begin{proof}
The proof of the result in the $p$-variation context is  developed in \cite[Theorem 5.1]{Gassiat24}, the H\"older case follows from the same proof verbatim.
\end{proof}
By simply taking for example $\mathcal{H} = H^2_0$ the space of functions with second derivative in $L^2$ starting from $0$, which satisfies $\mathcal{H}  \subset C^1$, and introducing $\mathfrak{g}^N(\mathbb{R}^d)$ the step-$N$ free nilpotent Lie algebra associated to $G^{N}(\mathbb{R}^d)$ which is identified to $T_{\mathbf{1}} G^{N}(\mathbb{R}^d)$, we obtain the following result, which we will be a fundamental tool in the proof of \Cref{thm_main}. 
\begin{corollary} \label{cor:fundamental_step}
For any  \( \mathbf{X} \in \mathcal{C}^{p\text{-var}}_{\mathbf{1}} \) there exists \( h \in C^{1\text{-var}}_0 \) such that $
(\mathbf{X} \boxplus h)_1 = \mathbf{1}.$
Moreover, $h$ can be chosen to be non-singular, that is there exists an Hilbert space $\mathcal{H}\subset C^{1\text{-var}}$ such that $h\in \mathcal{H}$  and one has
\begin{equation}\label{eq:non-singular}
\langle \nabla_{\mathcal{H}} (\mathbf{X} \boxplus h)_1, \cdot \rangle \colon \mathcal{H} \to\mathfrak{g}^{\lfloor p \rfloor}(\mathbb{R}^d) \quad \text{is surjective}\,.
\end{equation}
In case when $ \mathbf{X}\in \mathcal{C}^{\alpha}_{\mathbf{1}} $  the path $h$  can be taken such that $h\in  C^{\alpha, 1\text{-var}}_0$.
\end{corollary}
\subsection{Stability of local surjectivity under perturbations}

In addition to Corollary \ref{cor:fundamental_step}, we use a classical lemma on the stability of the local surjectivity of a differentiable function under $C^1$ perturbations in the context of $G^N(\bR^d)$-valued maps. Although this topic is covered in the standard literature on nonlinear analysis in metric spaces (see, e.g. \cite{implicit_book}), we include it for completeness and to establish a common notation.

Both $G^N(\mathbb{R}^d)$ and $\mathfrak{g}^N(\mathbb{R}^d)$ can be equipped with the norm $\|a\| := \sum_{k=1}^N \|\pi_k(a)\|^{1/k}$ given before. Moreover, $G^N(\mathbb{R}^d)$ and $\mathfrak{g}^N(\mathbb{R}^d)$ are linked by the exponential and logarithm maps
\[
\exp : \mathfrak{g}^N(\mathbb{R}^d) \to G^N(\mathbb{R}^d)\,,\quad
\log : G^N(\mathbb{R}^d) \to \mathfrak{g}^N(\mathbb{R}^d)\,,
\]
which define global diffeomorphisms. For a $C^1$ map $f:\bR^n \to G^N(\bR^d)$ we denote by $\mathrm{d}f_x$ its differential and we introduce the left logarithmic derivative
\[
\mathrm{d}^* f_x := (L_{f(x)^{-1}})_* \mathrm{d}f_x,
\quad 
\mathrm{d}^* f_x v = \left.\frac{\mathrm{d}}{\mathrm{d}t} f(x)^{-1}\otimes f(x+tv)\right|_{t=0},
\]
via the left multiplication
\[
L_{f(x)^{-1}} : g\in G^N(\mathbb R^d)\mapsto f(x)^{-1}g\in G^N(\mathbb R^d).
\]
By construction $\mathrm{d}^* f_x$ is a linear map $\mathrm{d}^* f_x\colon \mathbb{R}^n \to \mathfrak{g}^N(\mathbb{R}^d)$. For $f,g \in C^1(\mathbb{R}^n, G^N(\mathbb{R}^d))$ and $R>0$, we define
\begin{equation}\label{eq_loc_C1_norm}
d_1^R(f,g) := \sup_{x \in B(0,R)} \|f(x)-g(x)\|
+ \sup_{x \in B(0,R)} \|\mathrm{d}^* f_x - \mathrm{d}^* g_x\|_{\text{op}}\,,
\end{equation}
with $B(0,R)$ the usual Euclidean closed 	ball centered at $0$ of radius $R>0$ and $\|A\|_{\text{op}}$ the operator norm of linear maps $A\colon\mathbb{R}^n \to \mathfrak{g}^N(\mathbb{R}^d)$ where we use on $\mathbb{R}^n$ the euclidean norm (which we denote by $\|\cdot\|$ with a slight abuse of notation). Note that $d_1^R(f,g)$ is monotone in $R$. We pass then to the main local surjectivity property.

\begin{proposition}\label{inversion}
Let $f\colon \mathbb{R}^n \to G^N(\mathbb{R}^d)$ be a $C^1$ function such that
$f(0)=\mathbf{1}$ and $\mathrm{d}^* f_0=\mathrm{d}f_0$ is invertible. Then there exist $R,\varepsilon>0$ such that for every $g \in C^1(\mathbb{R}^n,G^N(\mathbb{R}^d))$ satisfying
$ d_1^R(f,g) \le \varepsilon$,
one has $\mathbf{1}\in g\big(B(0,R/2)\big)$.
\end{proposition}

\begin{proof}
To prove the result we will use the global chart $\log$ to restate the theorem. Setting $F := \log \circ f$ and $G := \log \circ g$, the condition
$\mathbf{1} \in g(B(0,R/2))$ is equivalent to $F(0)=0$ and we search $x^* \in B(0,R/2)$ with
$G(x^*)=0$. Since $\log$ is smooth and for any $h\colon \mathbb{R}^n \to G^N(\mathbb{R}^d)$ one has the explicit formula on $H= \log \circ h$
\[dH_x=\frac{\text{ad}_{H(x)}}{e^{\text{ad}_{H(x)}}- \text{id}}\circ \mathrm{d}^* h_x=\sum_{k=0}^{N-1}\frac{B_k}{k!}\text{ad}^k_{H(x)}\circ \mathrm{d}^* h_x\]
where $\{B_k\}_{k\geq 1}$ are the Bernoulli number with the convention $B_1=- \frac{1}{2}$ and $\text{ad}_{H(x)}$ is the usual adjoint map, the condition $d_1^R(f,g)\leq \varepsilon$ implies
\[
    \sup_{x\in B(0,R)}\|G(x) - F(x)\| +    \sup_{x\in B(0,R)}\|\mathrm{d}G_x - \mathrm{d}F_x\|_{\text{op}} \leq C_{R,N}\,\varepsilon
\]
for some constant $C_{R,N} \geq 1$ depending only on $N$ and $R$. Replacing $\varepsilon$ by $C_{R,N}\varepsilon$ if necessary, we may assume
\begin{equation}\label{eq:closeness}
    \sup_{x\in B(0,R)}\|G(x) - F(x)\| +    \sup_{x\in B(0,R)}\|\mathrm{d}G_x - \mathrm{d}F_x\|_{\text{op}}\leq \varepsilon\,.
\end{equation}
Recalling that $\mathrm{d}F_0= \mathrm{d}^*f_0=\mathrm{d}f_0 $, by continuity of $\mathrm{d}F$ at $0$ we can choose $0<R \leq 1$ small enough  so that
\begin{equation}\label{eq:R}
    \sup_{x \in B(0,R)} \|\mathrm{d}F_x - \mathrm{d}f_0\|_{\text{op}}
    \leq \frac{1}{4\|\mathrm{d}f_0^{-1}\|_{\text{op}}}\,,
\end{equation}
which is possible by continuity. Set $\varepsilon := R/(4\|(\mathrm{d}f_0)^{-1}\|)$. We claim that the map
\[
    T(x) := x - ((\mathrm{d}f_0)^{-1})G(x)
\]
is a contraction on $B(0,R)$ mapping into $B(0,R/2)\subset B(0,R)$, so that its unique fixed point yields the desired zero of $G$.  Indeed, for
$x\in B(0,R)$, one has 

\[T(x) = \mathrm{d}f_0^{-1}(\mathrm{d}f_0\,x - G(x))= \mathrm{d}f_0^{-1}\left((\mathrm{d}f_0\,x - F(x)) + (F(x)-G(x))\right)\]

By the mean value inequality applied to $x\mapsto F(x)-\mathrm{d}f_0\,x$, which vanishes at $0$ together with  the property \eqref{eq:R}, one has for all $x\in B(0,R)$
\begin{equation}\label{eq:MVT}
    \|F(x) - \mathrm{d}f_0\,x\| \leq \frac{\|x\|}{4\|(\mathrm{d}f_0)^{-1}\|_{\text{op}}}\,.
\end{equation}
Then from estimates \eqref{eq:MVT} and \eqref{eq:R} we obtain
\[
    \|T(x)\| \leq \|(\mathrm{d}f_0)^{-1}\|_{\text{op}}\!\left(\frac{R}{4\|(\mathrm{d}f_0)^{-1}\|_{\text{op}}}
    + \varepsilon\right) = \frac{R}{2},
\]
For the contraction property, take  $x,y\in B(0,R) $ and $z\in B(0,R)$. From  \eqref{eq:R} and \eqref{eq:closeness} one has
\[
    \|\mathrm{d}f_0 - \mathrm{d}G_z\|
    \leq \|\mathrm{d}f_0 - \mathrm{d}F_z\| + \|\mathrm{d}F_z-\mathrm{d}G_z\|
    \leq \frac{1+R}{4\|(\mathrm{d}f_0)^{-1}\|_{\text{op}}},
\]
and therefore, writing \[T(x)-T(y) = (\mathrm{d}f_0)^{-1}\int_0^1(\mathrm{d}f_0 -
\mathrm{d}G_{y+t(x-y)})(x-y)\,\mathrm{d}t\,,\]
we obtain
\[
    \|T(x)-T(y)\| \leq \frac{1+R}{4}\,\|x-y\| \leq \frac{1}{2}\,\|x-y\|,
\]
where the last inequality uses $R\leq 1$. By the Banach fixed point theorem, $T$ has a
unique fixed point $x^*\in B(0,R) $, which by the self-mapping property satisfies $x^*\in B(0,R/2)$ and $G(x^*)=0$, i.e.\ $g(x^*)=\mathbf{1}$.
Hence $\mathbf{1}\in g(B(0,R/2))$.
\end{proof}

\section{Convergence of translations} \label{sec:3}

We now proceed to the proof of Theorem~\ref{thm_main}. The sequence of paths appearing in~\eqref{eq:main_conv} will be constructed by concatenating several segments, each obtained by applying Corollary~\ref{cor:fundamental_step} on suitably chosen subintervals. To make this construction rigorous, we first establish suitable a priori bounds for a single control given by Corollary~\ref{cor:fundamental_step}. \subsection{Apriori estimates on the control}

We first introduce a quantitative indication over all possible controls in Corollary~\ref{cor:fundamental_step}.
\begin{definition}
For any \( \mathbf{X} \in \mathcal{C}^{p\text{-var}}_{\mathbf{1}} \) we set
$$
C(\mathbf{X}) = \inf\{\| h \|_{1\text{-var}} \colon (\mathbf{X} \boxplus h)_{1} = \mathbf{1}\}.
$$
Moreover in case $\mathbf{X} \in \mathcal{C}^{\alpha} $ we replace $\| h \|_{1\text{-var}}$ with $\| h \|_{\alpha;1\text{-var}}$ to obtain $C_{\alpha}(\mathbf{X})$.
\end{definition}

Thanks to the previous section $C(\mathbf{X})<\infty$ for any \( \mathbf{X} \in \mathcal{C}^{p\text{-var}}_{\mathbf{1}} \). Moreover, this quantity can be uniformly bounded in the strong rough path norm.

\begin{lemma}\label{finitude et linéarité}
	One has the  finiteness bound for any $p = \frac{1}{\alpha} > 1$ :
	\begin{equation}\label{finitude}
		\sup_{\|\mathbf{X}\|_{p\text{-var}}\leq  1}  C(\mathbf X) \leq \sup_{\|\mathbf{X}\|_{\alpha}\leq  1} C_{\alpha}(\mathbf X)<\infty\,.
	\end{equation}
\end{lemma}
\begin{proof}
The first inequality follows from the fact that for any continuous $p$-variation path, there exists a reparametrization of $[0,1]$ under which it becomes $1 /p$-H\"older (see \cite[Proposition 5.14]{frizbook}). It therefore suffices to prove the finiteness of the second supremum.


We prove the result by contradiction. Suppose that
\[
\sup_{\|\mathbf{X}\|_{\alpha}\leq 1} C_{\alpha}(\mathbf{X}) = \infty.
\]
Therefore we can construct a sequence $(\fX^m) \subset \cC^{{\alpha}}$ such that $C_\alpha(\fX^m) \to \infty$ and the property $\|\fX^m\|_{\alpha} \leq 1$ for all $m$. By the standard compactness criterion in \cite[Proposition 8.17]{frizbook}, up to extraction of a subsequence, $\fX^m$ converges to some $\fX \in \cC^{\alpha}$ in the $\rho_{\alpha'}$ metric for every $\alpha' < \alpha$ with $\|\fX\|_\alpha \leq 1$.

 Then we apply Corollary \ref{cor:fundamental_step} to  $\fX$ and consider  $h\in \mathcal{H}\subset C^{\alpha;1\text{-var}}_0$  such that $(\mathbf X\boxplus h)_{1}=\mathbf 1$. In addition to this, we can also exploit the property \eqref{eq:non-singular} and once fixed a basis of $e_1\,,\ldots \,,e_n$ of $\mathfrak{g}^{\lfloor\alpha^{-1} \rfloor}(\mathbb{R}^d) $ we fix $u_1\,,\ldots \,,u_n\in \mathcal{H}\subset C^{\alpha;1\text{-var}}_0$ such that for all $i=1\,, \ldots\,, n$
\begin{equation}\label{eq:surjectivity}
\langle \nabla_{\mathcal{H}} (\mathbf{X} \boxplus h)_1, u_i \rangle=e_i\,.
\end{equation}
These paths will allow to control the paths in the definition of $C_\alpha(\fX^m)$ in terms of the paths in the definition  of $C_\alpha(\fX)$.

 We define then for any $\fY\in  \cC^{\alpha}$ with $\|\fY\|_{\alpha} \leq 1$ the map $\Phi_\fY\colon \mathbb R^n\to G^{\lfloor \alpha^{-1} \rfloor}(\mathbb{R}^d) $ given by 
\[
	\Phi_\fY(x)=(\fY\boxplus (h+x\cdot u))_{1}\,,
\]
where $x\cdot u=x_1u_1+\dots x_n u_n$. By construction and using again Corollary \ref{cor:fundamental_step}, $\Phi_\fX$ is a $C^1$ map that satisfies $\Phi_\fX(0)=\mathbf{1}$ and $\mathrm{d}(\Phi_\fX)_0$ is invertible because from \eqref{eq:surjectivity} it coincides with the identity in the basis given by $e_1\,,\ldots \,,e_n$.  In addition to the properties of $\Phi_\fX$, we can show that for any $R>0$ and $\alpha' < \alpha$ such that $\lfloor \frac{1}{\alpha'} \rfloor=\lfloor \frac{1}{\alpha}  \rfloor $ there exists a constant $C_{R,\alpha'}>0$ such that for any $\fY \in \cC^{{\alpha}}_{\mathbf{1}}$ with $\|\fY\|_{\alpha} \le 1$, one has
\begin{equation}\label{eq:control_d_1}
d_1^R(\Phi_\fY,\Phi_\fX) \le C_{R,\alpha'} \, \rho_{\alpha'}(\fY,\fX),
\end{equation}
where $d_1^R$ is given in \eqref{eq_loc_C1_norm}. To prove \eqref{eq:control_d_1} one uses the local Lipschitz property in 
\eqref{eq_loc_Lip}, which implies
\begin{equation}\label{first_bound}
\sup_{x \in B(0,R)} \|\Phi_\fY(x)-\Phi_\fX(x)\|\lesssim \rho_{\alpha'}(\fY,\fX)
\end{equation}
where $\lesssim$ stands for inequality up to a constant depending on $R$ and $\alpha'<\alpha$ such that $\lfloor \frac{1}{\alpha'} \rfloor=\lfloor \frac{1}{\alpha}  \rfloor $ .

Concerning the diagonal derivative, we derive an explicit representation for \( d^*(\Phi_\fY)_x v \). Consider the RDE representation of \( \Phi_\fY(x) \) in \eqref{eq:RDE_translation}. Differentiating with respect to the parameter \(x\) in direction \(v \in \mathbb{R}^n\), which enters through the drift term \(x \cdot u\), and using differentiability of the It\^{o}--Lyons map with respect to such parameter perturbations (see \cite[Theorem 11.6]{frizbook}), we obtain $d^* (\Phi_\fY)_x v=J_1^x$, where \(J^x\) solves the linear \(\mathfrak{g}^{\lfloor \frac{1}{\alpha}  \rfloor}(\mathbb{R}^d)\)-valued RDE
\begin{equation}\label{Lie_RDE}
\mathrm{d}J_t^x = \sum_{i=1}^d  [J_t^x, \mathrm{d}\mathbf{Y}^i_t]
+  \sum_{i=1}^d [J_t^x, \mathrm{d}h^i_t + x\cdot \mathrm{d}u_{t}^i]
+ \sum_{i=1}^d v \cdot \mathrm{d}u_{t}^i,
\qquad J_0^x = 0.
\end{equation}
This RDE admits an explicit solution in terms of the adjoint representation of the group-valued path. Setting \(\fY_t^x := (\fY \boxplus h + x\cdot u)_t\), one verifies by direct computation that, for \(0 \le s \le t \le 1\), the map
\[
(s,t) \mapsto \mathrm{Ad}_{\mathbf{Y}_t^x \otimes (\mathbf{Y}_s^x)^{-1}}
\]
coincides with the flow of the homogeneous linear RDE
\begin{equation}\label{Lie_RDE_2}
\mathrm{d} L_t^x = \sum_{i=1}^d  [L_t^x, \mathrm{d}\mathbf{Y}^i_t]
+  \sum_{i=1}^d [L_t^x, \mathrm{d}h^i_t + x\cdot \mathrm{d}u_{t}^i],
\qquad L_0 = \mathrm{Id}_{\mathfrak{g}}.
\end{equation}
Since this is a linear RDE, its solution defines a two-parameter flow, and the adjoint action above provides its fundamental solution.

By the variation of constants formula, it follows that
\[
J_t^x = \int_0^t \mathrm{Ad}_{\mathbf{Y}_t^x \otimes (\mathbf{Y}_s^x)^{-1}} \bigl( v \cdot \mathrm{d}u_s \bigr),
\]
where the integral is well-defined as a Young integral since \(u\) has finite \(1\)-variation.

By simply applying Young inequality one has 
\begin{align*}
\|\mathrm{d}^*(\Phi_{\mathbf{X}})_x - \mathrm{d}^*(\Phi_{\mathbf{Y}})_x\|_{\mathrm{op}}&= \sup_{\|v\|\leq 1}\| \int_0^1\left(\mathrm{Ad}_{\mathbf{Y}_1^x\otimes (\mathbf{Y}_s^x)^{-1}} -\mathrm{Ad}_{\mathbf{X}_1^x\otimes (\mathbf{X}_s^x)^{-1}}\right) (v\cdot \mathrm{d}u_{s})\|\\&
\le \sup_{s\in[0,1]} \|\mathrm{Ad}_{\mathbf{Y}_1^x\otimes (\mathbf{Y}_s^x)^{-1}} -\mathrm{Ad}_{\mathbf{X}_1^x\otimes (\mathbf{X}_s^x)^{-1}}\|_{\mathrm{op}(\mathfrak{g}^{\lfloor \alpha^{-1} \rfloor}(\mathbb{R}^d))}
\sum_{k=1}^n \|u_k\|_{1\text{-var}}.
\end{align*}
Finally, observe that the adjoint process solves a linear RDE of the form \eqref{Lie_RDE_2}. Hence, by local Lipschitz continuity of the It\^{o}--Lyons map for linear RDEs \cite[Theorem 10.26]{frizbook}, applied uniformly in \((s,t)\), together with the stability estimate \eqref{eq_loc_Lip} for the translation map, we obtain
\[
\sup_{x\in B(0,R)}\sup_{s\in[0,1]} \|\mathrm{Ad}_{\mathbf{Y}_1^x\otimes (\mathbf{Y}_s^x)^{-1}} -\mathrm{Ad}_{\mathbf{X}_1^x\otimes (\mathbf{X}_s^x)^{-1}}\|_{\mathrm{op}(\mathfrak{g}^{\lfloor \alpha^{-1} \rfloor}(\mathbb{R}^d))}
\lesssim   \, \rho_{\alpha'}(\mathbf{Y}^x,\mathbf{X}^x)\lesssim \rho_{\alpha'}(\mathbf{X},\mathbf{Y})\,.
\]
Combining with the previous estimate yields
\begin{equation}\label{second_bound}
\sup_{x\in B(0,R)} 
\|\mathrm{d}^*(\Phi_{\mathbf{Y}})_x - \mathrm{d}^*(\Phi_{\mathbf{X}})_x\|_{\mathrm{op}}
\lesssim  C_R \, \rho_{\alpha'}(\mathbf{X},\mathbf{Y}).
\end{equation}
	and we obtain \eqref{eq:control_d_1}.	In this condition we are able to apply Proposition \ref{inversion}, which gives for our context the existence of $R,\eta >0$ such that if $\rho_{\alpha'}(\fY,\fX) \leq \eta $ then  there exists  $x\in B(0,R/2)$ such that $(\fY\boxplus h+x\cdot u)_{1}=\mathbf{1}\,$. By simply taking $m$ big enough so that $\rho_{\alpha'}(\fX^m,\fX)\leq \eta$ there exists $x^m\in B(0,R/2)$ such that $(\mathbf X^m\boxplus h+x^m\cdot u)_{1}=\mathbf 1$. Therefore
\[C_{\alpha}(\mathbf X^m)\leq \|h+x^m\cdot u\|_{\alpha,1\text{-var}}\leq n\frac{R}{2}\max_{i=1\,,\ldots \,, n} \|u_i\|_{\alpha,1\text{-var}}+\|h\|_{\alpha,1\text{-var}}\,.\]
So $C_{\alpha}(\mathbf X^m)$ is bounded and we provide a contradiction. 
\end{proof}
	From the boundedness property established in Lemma \ref{finitude}, we now derive a linear a priori control of \( C(\mathbf{X}) \). In particular, homogeneity together with boundedness immediately imply continuity of the function \(\mathbf{X}\mapsto C(\mathbf{X})\). 

\begin{proposition}\label{prop_apriori}
Let \( \mathbf{X} \in \mathcal{C}^{p\text{-var}}_{\mathbf{1}}\). Then for any \( r>0 \),
\begin{equation}\label{eq_homogeneity}
C(\delta_r \mathbf{X}) = r\, C(\mathbf{X})\,.
\end{equation}
Moreover, there exists a constant \( M>0 \) such that for all \( \mathbf{X} \in \mathcal{C}^{p\text{-var}} \),
\begin{equation}\label{eq_apriori}
C(\mathbf{X}) \le M \|\mathbf{X}\|_{p\text{-var}}\,.
\end{equation}
In case $\mathbf{X} \in \mathcal{C}^{\alpha}_{\mathbf{1}}$ the same result holds with $C_{\alpha}(\mathbf{X})$.
\end{proposition}

\begin{proof}
Thanks to Lemma \ref{finitude}, to prove \eqref{eq_apriori} is sufficient to show \eqref{eq_homogeneity}. Indeed supposing \eqref{eq_homogeneity} true, for any \( \mathbf{X}  \) such that $\|\mathbf{X}\|_{p\text{-var}} \neq 0$ one has
\[
C(\mathbf{X})
= \|\mathbf{X}\|_{p\text{-var}} \,
C\!\left(\delta_{\|\mathbf{X}\|_{p\text{-var}}^{-1}} \mathbf{X}\right),
\]
and since \( \|\delta_{\|\mathbf{X}\|_{p\text{-var}}^{-1}} \mathbf{X}\|_{p\text{-var}} = 1 \), the result follows from Lemma \ref{finitude}. To prove \eqref{eq_homogeneity}, let \( \mathbf{X} \in \mathcal{C}^{p\text{-var}}\) and $r>0$. we consider  \( h \) be an \(\varepsilon\)-minimizer for \(C(\mathbf{X})\), that is
\[
(\mathbf{X} \boxplus h)_1 = \mathbf{1}
\quad \text{and} \quad
\|h\|_{1\text{-var}} \le (1+\varepsilon) C(\mathbf{X})\,.
\]
We consider then $(x^n)\subset C^{1\text{-var}}$ a sequence of path approximating pointwise $\mathbf{X}$ from \eqref{eq:pointwise}. Using compatibility of dilation with the signature and translation, we compute
\begin{align*}
(\delta_r \mathbf{X} \boxplus r h)_1
&= \lim_{n\to\infty} (\delta_r S_N(x^n) \boxplus r h)_1= \lim_{n\to\infty} (S_N(rx^n) \boxplus r h)_1 \\
&= \lim_{n\to\infty} S_N\bigl(r(x^n + h)\bigr)_1 = \lim_{n\to\infty} \delta_r S_N(x^n + h)_1 = \delta_r (\mathbf{X} \boxplus h)_1= \mathbf{1}\,.
\end{align*}
Thus \( r h \) is admissible for \( \delta_r \mathbf{X} \), and
\[
C(\delta_r \mathbf{X}) \le \|r h\|_{1\text{-var}} \le r(1+\varepsilon) C(\mathbf{X}).
\]
Letting \(\varepsilon \to 0\) yields \( C(\delta_r \mathbf{X}) \le r C(\mathbf{X}) \). Applying the same argument to $\delta_r \mathbf{X}$ with \(r^{-1}\) instead of \(r\) gives the reverse inequality, and hence \eqref{eq_homogeneity}. The H\"older case follows verbatim.
\end{proof}
\subsection{Proof of the main theorem}
We can then provide a rigorous statement and a proof for Theorem \ref{thm_main}. Remarkably, it is sufficient to show precisely when $\mathbf{X}$ is generic and $\mathbf{Y}=\mathbf 1$ by simply patching together Theorem~\ref{thm_chow_transl} over a dyadic partition  of $[0,1]$.

\begin{theorem}\label{thm_main_rigorous}
	 For any \( \mathbf{X} \in \mathcal{C}^{p\text{-var}}_{\mathbf{1}} \)  there exists  a sequence of paths $(h^n)\subset  C^{\infty}_0$ such that $(\mathbf X\boxplus h^n)$ converges uniformly to $\mathbf{1}$ and for any $p'>p$ one has
	 \begin{equation}\label{eq:weak_conv}
\rho_{p'\text{-var}}((\mathbf X\boxplus h^n),\mathbf 1)\to 0 \,.
	\end{equation}
	Moreover if $\mathbf{X} \in \mathcal{C}^{\alpha}_{\mathbf{1}}$ the same result holds with the metric $\rho_{\alpha'} $ for any $\alpha'<\alpha$.
\end{theorem}
\begin{proof}
 For any fixed $\mathbf X \in \mathcal{C}^{p\text{-var}}_{\mathbf{1}} $   and $s,t\in[0,1]$, $s<t$ we localise translation on an interval $[s,t]$. Starting from the rescaled rough path $\mathbf{Y}^{s,t}_u=\mathbf X_{s,s+ (t-s)u}$ for $u\in [0,1]$, we apply Corollary \ref{cor:fundamental_step} and Proposition \ref{prop_apriori} to $\mathbf{Y}^{s,t}$ obtaining a path $g^{s,t}\in C^{1\text{-var}}_0 $ such that
\[
(\mathbf{Y}^{s,t}\boxplus g^{s,t})_1= \mathbf{1}\,, \quad \| g^{s,t}\|_{1\text{-var}}\leq 2 M \|\mathbf{X}\|_{p\text{-var}, [s,t]}\,.
\]
Defining then  
\[h(u)= g^{s,t}\left(\frac{u-s}{t-s}\right)\,, \quad u\in [s,t]\]
and constant elsewhere, we can easily obtain from \eqref{eq:coordinates_Young} and the invariance of the $1\text{-var}$ norm that $ h\in C^{1\text{-var}}([s,t]) $ and
\begin{equation}\label{eq_estimates} (\mathbf{X}\boxplus h)_{s,t}= (\mathbf{Y}^{s,t}\boxplus g^{s,t})_1= \mathbf{1}\,, \quad \| h\|_{1\text{-var}, [s,t]}=\| g^{s,t}\|_{1\text{-var}}\leq 2 M \|\mathbf{X}\|_{p\text{-var}, [s,t]}\,.
\end{equation}
 Now for any $n\geq 1$ take $\cD_n=\{k/2^n\colon k=0,\ldots, 2^n\}=\{t_k\colon k=0,\ldots, 2^n\}$ the set of $n$-dyadic numbers of $[0,1]$. By taking  $h^{(k)}$ define as before but for $(s,t)=(t_k,t_{k+1})$ and concatenating all of them, according to the dyadic meshing, we form a path $h^{n}$ with the following property: for all $k,l\in \{0,\ldots, 2^n\}$, $k\leq l$
\begin{equation}\label{eq_identity}
(\mathbf X\boxplus h^n )_{t_k,t_l}=\mathbf 1\,.
\end{equation}

	We now want to prove that $h^n$ is a good candidate for the convergence result. Let \(\mathbf Z^n:=\mathbf X\boxplus h^n\). By the local estimate for Young translation and by the construction of \(h^n\),
\begin{equation} \label{eq:313}
\|\mathbf{Z}^n\|_{p\text{-var};I_k}
\lesssim
\|\mathbf X\|_{p\text{-var};I_k}
+
\|h^n\|_{1\text{-var};I_k}
\lesssim
\|\mathbf X\|_{p\text{-var};I_k}.
\end{equation}

Let \(\Pi=\{0=u_0<\cdots<u_m=1\}\) be an arbitrary partition of \([0,1]\). For
\(u_i<1\), let \(k(i)\) be the unique index such that
\[
u_i\in [t_{k(i)},t_{k(i)+1}),
\]
and set \(k(m)=2^n-1\) if \(u_m=1\).

We split the increments of \(\Pi\) into two classes. First, consider those intervals
\([u_i,u_{i+1}]\) which are contained in a single dyadic cell. For each fixed \(k\), the
subcollection of such intervals contained in \(I_k\) is a subpartition of \(I_k\). Hence
\[
\sum_{\substack{i:\,[u_i,u_{i+1}]\subset I_k}}
|\mathbf{Z}^n_{u_i,u_{i+1}}|^p
\le
\|\mathbf{Z}^n\|_{p\text{-var};I_k}^p .
\]
Summing over \(k\), we obtain
\begin{equation} \label{eq:314}
\sum_{\substack{i:\,[u_i,u_{i+1}]\text{ contained}\\ \text{in one dyadic cell}}}
|\mathbf{Z}^n_{u_i,u_{i+1}}|^p
\le
\sum_{k=0}^{2^n-1}
\|\mathbf{Z}^n\|_{p\text{-var};I_k}^p .
\end{equation}

We now consider the remaining intervals, namely those which cross at least one dyadic
boundary. If \(k(i)<k(i+1)\), then by Chen's relation and by the identity
\(\mathbf{Z}^n_{t_{k(i)+1},t_{k(i+1)}}=\mathbf 1\), we have
\[
\mathbf{Z}^n_{u_i,u_{i+1}}
=
\mathbf{Z}^n_{u_i,t_{k(i)+1}}
\otimes
\mathbf{Z}^n_{t_{k(i+1)},u_{i+1}} .
\]
Therefore, by subadditivity of the Carnot--Carath\'eodory norm,
\begin{equation} \label{eq:315}
|\mathbf{Z}^n_{u_i,u_{i+1}}|^p
\lesssim
|\mathbf{Z}^n_{u_i,t_{k(i)+1}}|^p
+
|\mathbf{Z}^n_{t_{k(i+1)},u_{i+1}}|^p .
\end{equation}
For each dyadic cell \(I_k\), there is at most one crossing interval which starts in
\(I_k\), and hence at most one right-edge interval of the form
\[
[u_i,t_{k+1}]\subset I_k.
\]
Similarly, there is at most one crossing interval which ends in \(I_k\), and hence at most
one left-edge interval of the form
\[
[t_k,u_{i+1}]\subset I_k.
\]
Consequently, after summing \eqref{eq:315} over all crossing intervals,
\begin{equation} \label{eq:316}
\sum_{\substack{i:\,[u_i,u_{i+1}]\text{ crosses}\\ \text{a dyadic boundary}}}
|\mathbf{Z}^n_{u_i,u_{i+1}}|^p
\lesssim
2\sum_{k=0}^{2^n-1}
\|\mathbf{Z}^n\|_{p\text{-var};I_k}^p .
\end{equation}

Combining \eqref{eq:314} and \eqref{eq:316}, we get
\begin{equation} \label{eq:pvar}
\sum_{i=0}^{m-1}
|\mathbf{Z}^n_{u_i,u_{i+1}}|^p
\lesssim
\sum_{k=0}^{2^n-1}
\|\mathbf{Z}^n\|_{p\text{-var};I_k}^p .
\end{equation}
Using (3.13), this yields
\[
\sum_{i=0}^{m-1}
|\mathbf{Z}^n_{u_i,u_{i+1}}|^p
\lesssim
\sum_{k=0}^{2^n-1}
\|\mathbf X\|_{p\text{-var};I_k}^p
\le
\|\mathbf X\|_{p\text{-var};[0,1]}^p .
\]
Taking the supremum over all partitions \(\Pi\), we conclude that
\[
\|\mathbf{Z}^n\|_{p\text{-var};[0,1]}
\lesssim
\|\mathbf X\|_{p\text{-var};[0,1]},
\]
uniformly in \(n\).

Applying the compactness criterion \cite[Proposition 8.17]{frizbook} the sequence $\mathbf Z^n$ converges up to subsequences to some rough path $\mathbf{Z}$ in the $\rho_{p'}$ distance for any $p'>p$. Looking at the dyadics it is then easy to prove that $(\mathbf Z^n)$ converges punctually to $\mathbf 1$. Therefore  $\mathbf{Z}=\mathbf 1$ and we obtain \eqref{eq:weak_conv}.

{Note that so far the $h^n$ are only $C^1$, but taking $(h^{n,m})_{m \geq 0}$ smooth approximations of $h^n$, one has for each $n$ that $\|h^{n,m} - h^n\|_{1\text{-var}:[0,1]} \to 0$. In particular, by continuity of Young translation, we can choose $m(n)$ for which $\rho_{p\text{-var}}\left( \mathbf X\boxplus h^n, \mathbf X\boxplus h^{n,m(n)}\right) \leq 2^{-n}$, and we can then replace $h^n$ by $h^{n,m(n)} \in C^{\infty}$ while keeping the same convergence properties.}

 In case $\mathbf{X}\in\mathcal{C}^{\alpha}$ we can use the same argument to bound $\|\mathbf X\boxplus h^n\|_{\alpha}$ starting from a single increment and conclude accordingly.
\end{proof}
If the  initial rough path $\mathbf{X}$ is  a geometric rough path rather than a merely weakly geometric one, the following theorem provides a stronger result.

\begin{theorem}\label{thm_main_rigorous_geo}
	For any  $\fX\in \cC^{0;p\text{-var}}_{\mathbf{1}}$  the sequence $(h^n)$ in Theorem \ref{thm_main_rigorous} can be taken such that \begin{equation}\label{eq:strong_conv}
\rho_{p\text{-var}}((\mathbf X\boxplus h^n),\mathbf 1)\to 0\,.
	\end{equation}
	Moreover if $\mathbf{X} \in \mathcal{C}^{0;\alpha}_{\mathbf{1}}$ the same result holds with the metric $\rho_{\alpha} $.
\end{theorem}
\begin{proof}
In order to prove the strong convergence result, we will essentially use what was proved previously and use the same notation. Furthermore, we will use the vanishing characterization of geometric rough paths given in \cite[Theorem 8.22]{frizbook}.

Since any geometric rough path $\fX$ is weakly geometric, we obtain the same results as in the proof of Theorem \ref{thm_main_rigorous} up to the bound \eqref{eq:pvar}.  By the vanishing characterization of geometric rough paths, for every integer $m \geq 1$, there exists an integer $n(m)$ such that, whenever the mesh of the partition $\mathcal D_n$ is sufficiently small,
\[
\forall n\geq n(m), \qquad
\sum_{k=0}^{2^n-1}
\|\mathbf X\|_{p\text{-var};[t_k,t_{k+1}]}^p
\leq \frac{1}{m}.
\]

Repeating the argument used in the proof of Theorem \ref{thm_main_rigorous}, we obtain the bound
\[
\forall n \geq n(m), \qquad
\|\fX\boxplus h^n \|_{p\text{-var}}
\lesssim \frac{1}{m^{1/p}}.
\]

Therefore, possibly after passing to a subsequence, we obtain the stronger convergence result
\[
\rho_{p\text{-var}}(\fX\boxplus h^n, \mathbf 1)
\to 0.
\]
As before, we can approximate the $C^1$ path $h^n$ by a smooth path to conclude the proof.
\end{proof}

By simply combining the  Theorems \ref{thm_main_rigorous} and \ref{thm_main_rigorous_geo} with the usual approximations property of weak geometric rough paths with signatures, we show a precise version of Theorem~\ref{thm_main}. 

\begin{theorem}\label{thm_main_result}
Let $\mathbf{X}, \mathbf{Y} \in \mathcal{C}^{p\text{-var}}_{\mathbf{1}}$. Then there exists  a sequence of paths $(h^n)\subset  C^{\infty}_0$ such that $(\mathbf X\boxplus h^n)$ converges uniformly to $\mathbf{Y}$ and for any $p'>p$ one has
	 \begin{equation}\label{eq:weak_conv_Y}
\rho_{p'\text{-var}}(\mathbf X\boxplus h^n,\mathbf{Y} )\to 0 \,.
	\end{equation}
	Moreover if $\mathbf{X}, \mathbf{Y}\in \mathcal{C}^{0,p\text{-var}}_{\mathbf{1}}$ the sequence $(h^n) $ can be chosen  such that
	 \begin{equation}\label{eq:strong_conv_Y}
\rho_{p\text{-var}}(\mathbf{X} \boxplus h^n, \mathbf{Y}) \to 0\,.
\end{equation}
If in addition $\mathbf{X}, \mathbf{Y} \in \mathcal{C}^{\alpha}$, then $(h^n)$ can be chosen such that \eqref{eq:weak_conv_Y} and \eqref{eq:strong_conv_Y} hold with the metrics $\rho_{\alpha'}$ for any $\alpha'<\alpha$, and $\rho_{\alpha}$ in the case $\mathbf{X}, \mathbf{Y} \in \mathcal{C}^{0;\alpha}_{\mathbf{1}}$.
\end{theorem}

\begin{proof}
We prove only \eqref{eq:weak_conv_Y}, since the other statements can be established by the same argument.\ By Theorem~\ref{thm_main_rigorous_geo}, there exists a sequence $(k^n)\subset  C^{\infty}_0$ such that $(\mathbf X\boxplus k^n)$ converges uniformly to $\mathbf{1}$ and for any $p'>p$ one has
\[
\rho_{p'\text{-var}}(\mathbf{X} \boxplus k^n, \mathbf{1}) \to 0.
\]
Since $\mathbf{Y} \in \mathcal{C}^{p\text{-var}}_{\mathbf{1}}$, from property \eqref{eq:approx} there exists a sequence $(y^m)\subset  C^{\infty}_0$ such that $S(y^m)\to \mathbf{Y}$ uniformly and
\[
\rho_{p'\text{-var}}(S(y^m), \mathbf{Y}) \to 0.
\]
Using that $\mathbf{1} \boxplus y^m = S(y^m)$ and the additivity of Young translation, we have
\[
\mathbf{X} \boxplus (k^n + y^m)
=
(\mathbf{X} \boxplus k^n) \boxplus y^m.
\]
We fix $m$. By continuity of the map $\mathbf{Z} \mapsto \mathbf{Z} \boxplus y^m$ in $p$-variation, as $n \to \infty$
\[
(\mathbf{X} \boxplus k^n) \boxplus y^m \to \mathbf{1} \boxplus y^m = S(y^m)\,.
\]
Hence, for each $m\geq 1$, there exists $n(m)$ such that
\[
\rho_{p'\text{-var}}(\mathbf{X} \boxplus (k^{n(m)} + y^m), S(y^m)) < \frac{1}{m}.
\]
Therefore if we define
\[
h^m := k^{n(m)} + y^m,
\]
we obtain \eqref{eq:weak_conv_Y}.
\end{proof}

\section{Applications to support theorems}\label{sec:4}

%
%
%
%
%
Fix $\alpha \in (0,1]$, and $q \geq 1$ with $\alpha + \frac{1}{q} > 1$. Given a subset $\mathcal{F}$ of $\mathcal{C}^{0;\alpha}_{\mathbf{1}}$, we define for any $h \in C^{0;\alpha,q\text{-var}}_0$, the Young translation on the set $\mathcal{F}$ by  $\mathcal{F}\boxplus h = \{ \fX \boxplus h, \; \fX \in \mathcal{F}\}$ and then
\begin{equation} \label{eq:defK}
\cK_{\alpha,q}({ \mathcal{F}}) = \left\{ h \in C^{\alpha,q\text{-var}}_0 : \mathcal F\boxplus h\subset \mathcal F \right\}.
\end{equation}
Similarly for $F \subset C^{\alpha,q\text{-var}}_0$, we will also denote $K_{\alpha,q}(F) = \{h \in C^{0;\alpha,q\text{-var}}_0: h+F \subset F \}$. Note that for any $\mathcal{F}$, $\cK_{\alpha,q}({ \mathcal{F}})$ is stable under addition thanks to the property of the translation. As a corollary of our main result, we immediately have the following sufficient condition for a closed set to be the whole space.

\begin{corollary} \label{cor:main}
Let $\mathcal{F}$ be a closed subset of $\mathcal{C}^{0;\alpha}_{\mathbf{1}}$, $\mathcal{F} \neq \emptyset$, and such that, for some $q$ with $\frac{1}{q}+\alpha>1$,  $\cK_{\alpha,q}({\mathcal{F}})$ is dense in $C^{0;\alpha,q\text{-var}}$. Then $\mathcal{F}= \mathcal{C}^{0;\alpha}_{\mathbf{1}}$.
\end{corollary}

\begin{proof}
We fix $\mathbf{X} \in \mathcal{F}$, and let $\mathbf{Y}$ be an arbitrary element in $\mathcal{C}^{0;\alpha}_{\mathbf{1}}$. By the main theorem, there exists $h_n \in C^{\infty}_0$ such that $ \mathbf{X} \boxplus h_n\to  \mathbf{Y}$.  Let $h'_n \in \cK_{\alpha,q}(\mathcal{F})$ s.t. $\|h'_n - h_n \|_{C^{\alpha,q\text{-var}}}  \to 0$. We have
\[
(\mathbf{X}\boxplus h'_n) = (\mathbf{X}\boxplus h_n)\boxplus (h'_n - h_n) \to \mathbf{Y}
\]
by the joint continuity of Young translation. In addition, by definition of $\cK_{\alpha,q}(\mathcal{F})$, $\mathbf{X} \boxplus h'_n\in \mathcal{F}$ for each $n$. Since $\mathcal{F}$ is closed, it must contain $\mathbf{Y}$.

\end{proof}

As further corollaries, we will get the full support property for certain rough path-valued random variables.

\subsection*{Random series}
We first have an application to random series.
\begin{corollary} \label{cor:series}
Let $X_n$ be independent path-valued random variables, with values in $C^{0;\alpha,q\text{-var}}$ for some $\alpha \in (0,1]$ and $\frac{1}{q}+\alpha>1$. Assume that $\mathbf{X}$ is a $\mathcal{C}^{0;\alpha}_{\mathbf{1}}$-valued random variable, such that
 \[
 \lim_{N \to \infty} S_{\lfloor \alpha^{-1} \rfloor}\left( \sum_{n \leq N} X_n \right) = \mathbf{X} \mbox{ in probability, in }\mathcal{C}^{0,\alpha}_{\mathbf{1}}.
 \]
For each $n$, let $S_n \subset C^{0;\alpha,q\text{-var}}_0$ be the support of the law of $X_n$ in $C^{\alpha,q\text{-var}}_0$,  let $K_n =K_{\alpha,q} (S_n)$, and assume that $\sum_{n \geq 0} K_n $ (the set of finite sums of elements of the $K_n$) is dense in $C^{0;\alpha,q\text{-var}}$. Then, the law of $\mathbf{X}$ has full support in $\mathcal{C}^{0;\alpha}_{\mathbf{1}}$.
\end{corollary}

\begin{proof}

Let $\mathcal{S}$ be the support of the law of $\mathbf{X}$ in $\mathcal{C}^{0;\alpha}_{\mathbf{1}}$. Note that since $\mathcal{C}^{0;\alpha}_{\mathbf{1}}$ is a Polish space, $\mathcal{S}$ is nonempty. Fix $m \geq 0$, and note that one has $\mathbf{X} =  \mathbf{Y}^m \boxplus X_m$, where 
\[\mathbf{Y}^m = \mathbf{X} \boxplus (- X_m) = \lim_{N\to\infty} S\left( \sum_{n \leq N, n \neq m} X_n\right)\] is independent from $X_m$. By continuity of Young translation, it follows that
\[
\mathcal{S} = \overline{ \left\{  (\mathbf{Y}\boxplus x), \mathbf{Y} \in Supp(Law(\mathbf{Y}^{m})) , x \in \cS_m  \right\}}.
\]
The above set is stable under Young translation along $h$ for any $h \in \cK_m$, and it follows that $K_m \subset \cK_{\alpha,q}(\mathcal{S})$ for any $m$, and therefore $\sum_m K_m \subset \cK_{\alpha,q}(\mathcal{S})$. Since $ \mathcal{S}$ is a nonempty closed set, and $\sum_m K_m$ is dense in $C^{0;\alpha,q\text{-var}}$, we conclude by Corollary \ref{cor:main}.

%
%
%
%
\end{proof}
\begin{remark}
This corollary implies for instance, that if
\[
X = \sum_{n} \gamma_n x_n,
\]
where $\gamma_n$ are independent random variables, each with full support on $\bR$, and the $x_n$ are a collection of smooth paths such that their span is dense in $C^1([0,T],\bR^d)$, then if $X$ admits a canonical rough path lift (i.e. as limit of lifts of partial sums), then its law has full support on rough path space.
\end{remark}
\subsection*{Gaussian rough paths}
We now consider the case of Gaussian rough paths. Given a $C([0,T],\bR^d)$-valued Gaussian process $X$, recall that its Cameron-Martin space $\cH_X$ is defined as those elements $h \in C([0,T],\bR^d)$ such that the law of $X+h$ is equivalent to that of $X$. It then holds that $\cH_X$ can be equipped with a norm $\|\cdot\|_{\cH_X}$ under which it is a Hilbert space, and such that for any $h \in \cH_X$, there exists a random variable $Z_h$ with distribution $\cN(0,\|h\|_{\cH_X}^2)$ such that
\[
\bE\left[F(X+ h) \right] = \bE\left[ F(X) \exp\left(Z_h - \frac 1 2 \|h\|_{\cH_X}^2\right)\right]
\]
for any Borel-measurable $F$ (this is the Cameron-Martin formula).

\begin{corollary}[Gaussian rough paths] \label{cor:gauss}
Fix  $\alpha \in (0,1]$ and $q \geq 1$ with $\frac{1}{q}+\alpha>1$. Let $X^n$ be a sequence of $C^{0;\alpha,q\text{-var}}$-valued Gaussian processes, and denote by  $\cH_{X^n}$ their Cameron-Martin spaces. Assume that $\mathbf{X}$ is a $\mathcal{C}^{0;\alpha}_{\mathbf{1}}$-valued random variable, such that
 \[
 \lim_{n \to \infty} S_{\lfloor \alpha^{-1} \rfloor}\left( X_n \right) = \mathbf{X} \mbox{ in probability, in }\mathcal{C}^{0,\alpha}_{\mathbf{1}}.
 \]
Further assume that
\[
\cH = \left\{ h = \lim_{n \to \infty} h_n \mbox{ in }C^{\alpha,q\text{-var}}, \;\; h_n \in \mathcal{H}_{X^n}, \sup_{n\geq 0} \|h_n\|_{\cH_{X^n}} < \infty \right\}
\]
is dense in $C^{0;\alpha,q\text{-var}}_0$. Then the law of $\mathbf{X}$ has full support in $\mathcal{C}^{0;\alpha}_{\mathbf{1}}$.
\end{corollary}

\begin{proof}
Let $\mathcal{S}$ be the support of the law of $\mathbf{X}$ in $\mathcal{C}^{0;\alpha}_{\mathbf{1}}$. By corollary \ref{cor:main}, it suffices to show that $\cH \subset \cK_{\alpha,q}(\mathcal{S})$. We fix $h \in \cH$, $h = \lim_{n \to \infty} h_n$ with $h_n$ as in the definition of $\mathcal{H}$. We take $\mathbf{Y}$ in $\mathcal{S}$ and aim to show that $\mathbf{Y}\boxplus h$ is again in $\mathcal{S}$. Fix $\epsilon >0$. By continuity of Young translation, there exists $\epsilon'>0$ s.t.
\[
P(d_{\alpha}(\mathbf{X},\mathbf{Y}\boxplus h) < \epsilon)\;\; \geq \;\;\bP(d_{\alpha}((\mathbf{X}\boxplus -h),\mathbf{Y}) < \epsilon').
\]
We further have, using joint continuity of Young translation in the first equality,
 \begin{align*}
\bP(d_{\alpha}(\mathbf{X}\boxplus -h,\mathbf{Y}) < \epsilon')\geq  & \liminf_{n \to \infty}\bP(d_{\alpha}(S(X_n-h_n),\mathbf{Y}) < \epsilon'/2)  \\
= & \liminf_{n \to \infty} \bE \left[ 1_{\{d_{\alpha}(S(X_n),\mathbf{Y}) < \epsilon'/2\}} Y^n \right] 
\end{align*}
by the Cameron-Martin formula, where $L^n= \exp( \|h_n\|_{\cH_{X^n}} Z_n - \frac{1}{2}  \|h_n\|_{\cH_{X^n}}^2)$ for some $Z_n \sim \cN(0,1)$. In particular $\bE[(Y^n)^{p}] = \exp((p^2-p)/2 \|h_n\|^2)$ for $p \in \bR$. We then use H\"older's inequality (in the form $\bE[XY] \geq \bE[X^{1/2}]^2 \bE[Y^{-1}]^{-1}$) to write that the above quantity satisfies
\begin{align*}
\liminf_{n \to \infty} \bE \left[ 1_{\{d_{\alpha}(S(X_n),\mathbf{Y}) < \epsilon'/2\}} L^n \right] \geq& \liminf_{n \to \infty}\bP(d_{\alpha}(S(X_n),\mathbf{Y}) < \epsilon'/2)^2 \bE[L_n^{-1}]^{-1} \\
\geq&\bP(d_{\alpha}(S(\mathbf{X}),\mathbf{Y}) < \epsilon'/4)^2 \exp( -  \limsup_{n \to \infty}\|h_n\|^2)
\end{align*}
which is strictly positive by assumption. In conclusion, $\mathbf{Y}\boxplus h \in \mathcal{S}$ for each $h \in \cH$, $\mathbf{Y} \in \mathcal{S}$, i.e. $\cH \subset \cK(\mathcal{S})$, and we conclude by Corollary \ref{cor:main}.
\end{proof}

\subsection*{Non-centered Fractional Brownian rough paths}

Fix $\frac 1 3 < \alpha < H \leq \frac 1 2$. Let $\fX=(x,\mathbb X)$ be a fixed element of $\mathcal{C}^{0;\alpha}_{\mathbf{1}}$, we want to consider the random rough path $\fB^H \widehat{\boxplus}\;  \fX$, the lift of fractional Brownian motion $B^H$ with Hurst index $H$, centered at $\fX$. Since this is not a standard object, we define it in the proposition below.
 
\begin{proposition}
Fix $\frac 1 3 < \alpha < H \leq \frac 1 2$, a fractional Brownian motion $B^H$ and $\fX$ in $\mathcal{C}^{0;\alpha}_{\mathbf{1}}$. Let $B_n^H$ be piecewise linear approximations of $B^H$ and $x_n$  $C^1$ paths with $S^2(x_n) \to \fX$ in $\cC^{0;\alpha}_{\mathbf 1})$. Then there exists a random ($\alpha$-H\"older) geometric rough path $\fB^H \widehat{\boxplus}\;  \fX$ such that
\begin{equation}
\fB^H \widehat{\boxplus}\;  \fX = \lim_n S^2(B_n^H +x_n),
\end{equation}
in probability and $\alpha$-H\"older rough path topology. In addition, the limit does not depend on the choice of the sequence $(x_n)$.
\end{proposition}

\begin{proof}
We first define the (random) joint lift  $(\fX,\fB^H)$ as a ($\alpha$-H\"older) geometric rough path with values in $\bR^d$.\ The case $H=1/2$ was treated by \cite{DOR15}, the case $H> 1/3$ is similar. The crucial point is to check that the mixed integral $\int_s^t (x_u - x_s) dB^H_u$ (since $x$ is deterministic, this is just a Wiener integral)  is a.s. of order $|t-s|^{2\alpha}$. Note that Wiener integrals against fractional Brownian motion satisfy 
\[
\left\| \int_0^1 f_s dB^H_s \right\|_{L^2(\Omega)} \leq C \|f \|_{C^{\gamma}}
\]
for any $\gamma > 1/2 - H$, see e.g. \cite[Theorem 1.9.9]{Mishura2008}. Since under our assumption $\alpha > 1/2-H$, we obtain by scaling that $\|\int_s^t (x_u - x_s) dB^H_u \|_{L^2(\Omega)} \leq C \|x\|_{\alpha} |t-s|^{\alpha+H}$, by Gaussianity this holds also in any $L^p$, $p<\infty$, and by the Kolmogorov criterion for rough paths, this implies that a.s. it is of order $|t-s|^{2\alpha}$.

In addition, it can be proven by straightforward arguments that this joint lift has the continuity properties that if $x_n$ are $C^1$ paths with $S_2(x_n) \to \fX$ in $\cC^{\alpha}_g$ and $B_n^H$ are piecewise linear (for instance) approximations of $B^H$, then the joint lift $S_{2}(x_n,B_n^H)$ converges to $(\fX,\fB^H)$ in probability and $\alpha$-H\"older rough path topology.

We then define $\fB^H \widehat\boxplus\;  \fX$  from the joint lift by $\fB^H +\fX = \textbf{plus}(\fX,\fB^h)$ where $\textbf{plus}$ is the natural extension of $(x,y) \in \bR^{2d} \mapsto x+y \in \bR^d$, see \cite{frizbook}, section 7.5. It holds in particular that
\[
 \fB^H \widehat\boxplus \; \fX = \lim_{n \to \infty} S_{2}(B^H_n+x_n)
\]
in $\mathcal{C}^{0,\alpha}_{\mathbf{1}} $ and in probability, where $x_n$ and $B^H_n$ are as in the statement of the theorem.
\end{proof}

We then have the following support property on the law of $\fB^H \widehat\boxplus \; \fX$. 

\begin{corollary}
For any $\frac 1 3 < \alpha < H \leq \frac 1 2$, any $\fX$ in $\mathcal{C}^{0;\alpha}_{\mathbf{1}}$, 
\[
\mbox{ The law of }\fB^{H} \widehat\boxplus\;  \fX \mbox{ has full support in } \mathcal{C}^{0;\alpha}_{\mathbf{1}}.
\]
\end{corollary}
\begin{proof}
We apply Corollary \ref{cor:gauss} with $X_n =  B^H_n + x_n$. Since the Cameron-Martin space of a Gaussian process does not depend on its mean, it holds that $\cH_{X_n} = \cH_{B^H_n}$.

Note that $\cH_{B^{H}} \subset C^{0;\alpha,q\text{-var}}$ for $\alpha < H$ and $q > 1/(H+1/2)$, see \cite{FV06embedding}, and it is a dense subset since it contains smooth functions, see e.g. \cite[Section 5.2]{Pic10}). 

In addition, for any $h \in \cH_{B^H}$, its piecewise linear approximations $h_n$ are in $ \cH_{B^H_n}$ with 
\[\|h_n\|_{ \cH_{B^H_n}} = \inf \left\{ \|k \|_{\cH_{B^H}}: k_n = h_n \right\} \leq \|h\|_{\cH_B^H}\]
(as a special case of the behaviour of Cameron-Martin spaces under continuous linear maps, see e.g. \cite[Proposition 3.71]{Hai09}).  It follows that all the assumptions of Corollary \ref{cor:gauss} are met and we can conclude.
\end{proof}

Note that $\fB^{H}\widehat \boxplus \; \fX$ it is a genuinely probabilistic object, depending on $\fX$, and cannot be in general written as a continuous function of $\fB^H$, so that the support theorem above is genuinely a new result and not a corollary of the full support property for the fBm lift $\fB^H$.

%
%
%

\bibliographystyle{alpha}
\bibliography{bibliography}

\end{document}